\documentclass[11pt]{article}
\usepackage{amsmath,amsthm, amssymb, amsfonts,accents,nccmath}
\usepackage{enumitem}
\usepackage{array}
\usepackage{lipsum}
\usepackage{mathtools}
\usepackage[toc,page]{appendix}
\usepackage[english]{babel}
\usepackage{dsfont}
\usepackage{booktabs}
\usepackage[linesnumbered,commentsnumbered,ruled,vlined]{algorithm2e}
\usepackage{tikz}
\usetikzlibrary{decorations.pathreplacing,calligraphy}
\usepackage{tikz-network}
\usetikzlibrary{hobby}
\usepackage[utf8]{inputenc}
\usepackage[english]{babel}

\usepackage{caption}
\usepackage{subcaption}

\usepackage{cite}

\newcommand\restr[2]{{
  \left.\kern-\nulldelimiterspace 
  #1 
  \vphantom{\big|} 
  \right|_{#2} 
  }}

\usepackage{hyperref}
\hypersetup{
    colorlinks=true,
    linkcolor=blue,
    filecolor=darkmagenta,      
    urlcolor=cyan,
    citecolor=red
} 
\makeatletter
\tikzset{
	dot diameter/.store in=\dot@diameter,
	dot diameter=2pt,
	dot spacing/.store in=\dot@spacing,
	dot spacing=9pt,
	dots/.style={
		line width=\dot@diameter,
		line cap=round,
		dash pattern=on 0pt off \dot@spacing
	}
}
\makeatother

\makeatletter
\def\BState{\State\hskip-\ALG@thistlm}
\makeatother
\numberwithin{equation}{section}

\newtheorem{theorem}{Theorem}[section]
\newtheorem{cor}[theorem]{Corollary}
\newtheorem{lem}[theorem]{Lemma}

\newtheorem{rem}[theorem]{Remark}

\newcommand{\mr}{\textup{mr}}

\title{On the spectral radius of clique trees with a given\\ zero forcing number }
\author{Joyentanuj Das\footnote{Department of Applied Mathematics,
National Sun Yat-sen University, Gushan District, Kaohsiung City, 804, Taiwan (ROC) \indent  Email: joyentanuj@gmail.com, joyentanuj@math.nsysu.edu.tw }}

\date{}

\begin{document}

\maketitle

\begin{abstract}
Let $G(n,k)$ be the class of clique trees on $n$ vertices and zero forcing number $k$, where $\left \lfloor \frac{n}{2} \right \rfloor + 1 \le k \le n-1$ and each block is a clique of size at least $3$. In this article, we proved the existence and uniqueness of a clique tree in $G(n,k)$ that attains maximal spectral radius among all graphs in $G(n,k)$. We also provide an upper bound for the spectral radius of the extremal graph.
\end{abstract}

\noindent {\sc\textsl{Keywords}:} complete graphs, clique trees, zero forcing number, spectral radius.

\noindent {\textbf{MSC}:}   05C50, 05C15, 15A18

\section{Introduction}\label{sec:intro}
Let $G=(V(G),E(G))$  be a finite, simple, connected graph with $V(G)$ as the set of vertices and $E(G)$ as the set of edges in $G$.  We write $u\sim v$ to indicate that the vertices $u,v \in V(G)$ are adjacent in $G$.   The degree of the vertex $v$, denoted by $d_G(v)$, equals the number of vertices in $V$ that are adjacent to $v$.   A graph $H$ is said to be a subgraph of $G$ if $V(H) \subset V(G)$ and $E(H) \subset E(G)$. For any subset $S \subset V (G)$, a subgraph $H$ of $G$ is said to be an induced subgraph with vertex set $S$, if $H$ is a maximal subgraph of $G$ with vertex set $V(H)=S$. We write $|S|$ to denote the cardinality of the set $S$. 

Let $G=(V(G),E(G))$ be a graph. For $u,v \in V$, the adjacency matrix of the graph  $G$ is,    ${A}(G) = [a_{uv}]$, where $a_{uv}= 1 $  if $u\sim v$ and $0$ otherwise.  For any column vector $\mathbf{x}$, if $x_u$  represent the entry of $\mathbf{x}$ corresponding to vertex $u\in V$, then 
$$\mathbf{x}^{\top} {A}(G)\mathbf{x}=  2\sum_{u\sim w}x_{u}x_{w}, $$
where $\mathbf{x}^{\top}$ represent  the transpose of  $\mathbf{x}$. 

For a connected graph $G$    on $n \geq 2$ vertices, by the Perron--Frobenius theorem, the spectral radius  $\rho(G)$ of ${A}(G)$ is a simple positive eigenvalue and the associated eigenvector is entry-wise positive (for details see~\cite{Bapat}). We will refer to such an eigenvector as the Perron vector of $G$.  Now we state a few known results on spectral radius useful in our subsequent calculations. By the Min-max theorem, we have
\begin{equation}\label{eqn:eq_sp_rd}
\rho(G) = \max_{\mathbf{x}\neq \mathbf{0}}\dfrac{\mathbf{x}^{\top} {A}(G)\mathbf{x}}{\mathbf{x}^{\top} \mathbf{x}}= \max_{\mathbf{x}\neq \mathbf{0}}\dfrac{2\sum_{u\sim w}x_{u}x_{w}}{\sum_{u\in V}x_{u}^2}.
\end{equation}
Furthermore, in Eqn.~\eqref{eqn:eq_sp_rd}, the maximum attained if and only if $\mathbf{x}$ is a Perron vector of $\rho(G).$

Given a graph $G=(V(G),E(G))$, for $u, v \in V(G)$ we will use  $G + uv$ to denote the graphs obtained from $G$  by adding an edge $uv \notin E(G)$ and we have the following result.
\begin{lem}\label{lem:sr_edge}\cite{Bapat}
	If $G$ is a graph such that for $u, v \in V(G)$, $uv \notin E(G)$, then $\rho(G) < \rho(G+uv)$. 
\end{lem}

Let $A$ be a real symmetric matrix whose rows and columns are indexed by $V = \{1,2,\cdots,n\}$. Let $\{V_1,V_2,\cdots,V_k\}$ be a partition of $V$ such that the block partition of the matrix $A$ according to $\{V_1,V_2,\cdots,V_k\}$ can be expressed as
\begin{equation*}\label{eqn:A}
	A = \begin{bmatrix}
		A_{11} & A_{12}       &  \dots & A_{1k} \\
		A_{21} & A_{22}       &  \dots & A_{2k} \\
		\vdots & \vdots & \ddots & \vdots\\
		A_{k1} & A_{k2}    &  \dots & A_{kk}
	\end{bmatrix},
\end{equation*}
where $A_{ij}$ denotes the block formed by intersection of the rows in $V_i$ and the columns in $V_j$. Let $q_{ij}(A)$ denote the average row sum of $A_{ij}$. Then, the quotient matrix of $A$ with respect to the partition $\{V_1,V_2,\cdots,V_k\}$ is given by $$Q(A) = [q_{ij}(A)].$$ Moreover, if the row sum of each block $A_{ij}$ is constant then we say that the partition is equitable and $Q(A)$ is called an equitable quotient matrix of $A$. There is a nice relation between the spectrum of $A$ and that of $Q(A)$, which is stated now as a theorem.

\begin{theorem}\label{thm:equitable}~\cite{You}
	Let $A$ be a real symmetric matrix such that it has an equitable quotient matrix $Q(A)$, then, $\sigma(Q(A)) \subset \sigma(Q(A))$. Moreover, if $A$ is nonnegative, then $\rho(A) = \rho(Q(A))$, \textit{i.e.}, the spectral radius of $Q(A)$ is actually the spectral radius of $A$. 
\end{theorem}

A vertex $v$ of a connected graph $G$ is a cut vertex of $G$ if $G - v$ is disconnected. A block of the graph $G$ is a maximal connected subgraph of G that has no cut-vertex. Given two blocks  $F$ and $H$  of graph $G$ are said to be adjacent if they are connected via a cut-vertex. We denote $F\circledcirc H$, to represent the induced subgraph on the vertex set of two adjacent blocks $F$ and $H$.

 A complete graph is a graph where each vertex is adjacent to every other vertex. A complete graph on $n$ vertices is denoted by $K_n$. A connected graph is called a clique tree if each of its blocks is a clique.  Let $G$ be a clique tree with $d_1$ blocks of $K_{n_1}$, $d_2$ blocks of $K_{n_2}$, so on up to $d_b$ blocks of $K_{n_b}$, then we write $G$ with blocks $ K_{n_1}^{(d_1)} - K_{n_2}^{(d_2)}- \cdots - K_{n_b}^{(d_b)}$ (for example see Figure~\ref{fig:block}). Here the above notation only gives information about the blocks of $G$, but not about the structure of the graph. But for a special case, if $G$ has a central cut vertex, then all the blocks are adjacent via the central cut vertex and we denote it by
$$G = K_{n_1}^{(d_1)} \circledcirc K_{n_2}^{(d_2)}\circledcirc \cdots \circledcirc K_{n_b}^{(d_b)}.$$

\begin{figure}[ht]
	\centering
	\begin{tikzpicture}[scale=1.0]
		\Vertex[size=.1,color=red]{A}
		\Vertex[y=2,size=.1,color=black]{B}
		\Vertex[x=2,y=0,size=.1,color=red]{C} 
		\Vertex[x=2,y=2,size=.1,color=black]{D}
		\Vertex[x=2,y=-2,size=.1,color=black]{E}
		\Vertex[x=4,y=0,size=.1,color=black]{F} 
		\Vertex[x=4,y=-2,size=.1,color=black]{G}
		
		\Vertex[x=3,y=2,size=.1,color=black]{D1}
		\Vertex[x=4,y=1,size=.1,color=black]{F1}

		\Vertex[x=-2,y=-.7,size=.1,color=black]{D2}
		\Vertex[x=-2,y=.7,size=.1,color=black]{F2}	
		
		\Vertex[x=-0.7,y=-2,size=.1,color=black]{D3}
		\Vertex[x=0.7,y=-2,size=.1,color=black]{E3}	
		
		\Edge[lw=2pt](A)(B)
		\Edge[lw=2pt](A)(C)
		\Edge[lw=2pt](A)(D)
		\Edge[lw=2pt](B)(C)
		\Edge[lw=2pt](B)(D)
		\Edge[lw=2pt](C)(D)
		
		\Edge[lw=2pt](C)(E)
		\Edge[lw=2pt](C)(F)
		\Edge[lw=2pt](C)(G)
		\Edge[lw=2pt](E)(F)
		\Edge[lw=2pt](E)(G)
		\Edge[lw=2pt](F)(G)
		
		\Edge[lw=2pt](C)(D1)
		\Edge[lw=2pt](C)(F1)
		\Edge[lw=2pt](D1)(F1)
		
		\Edge[lw=2pt](A)(D2)
		\Edge[lw=2pt](A)(F2)
		\Edge[lw=2pt](D2)(F2)
		
		\Edge[lw=2pt](A)(D3)
		\Edge[lw=2pt](A)(E3)
		\Edge[lw=2pt](E3)(D3)
	\end{tikzpicture}
\caption{A clique tree with blocks  $K_{3}^{(3)}-K_{4}^{(2)}$.} \label{fig:block}
\end{figure}
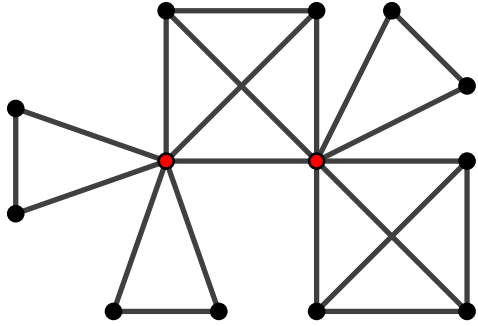

A block $H$ of a clique tree $G$ is a pendent block of $G$ if $H$ has exactly one cut-vertex of $G$. Let $v$ be a cut-vertex of $G$. If $G - v$ consists of two disjoint graphs $W_1$ and $W_2$ and let $G_i (i = 1, 2)$ be the subgraph of $G$ induced by $\{v\} \cup V(W_i)$, then $G$ is called the vertex-sum at $v$ of the two graphs $G_1$ and $G_2$, and denoted by $G = G_1 \oplus_v G_2$. For a vertex $v$ in a graph $G$, the block index of $v$ is the number of blocks which share the vertex $v$ and is denoted by $bi_G(v)$.

The spectral radius of a graph has been extensively studied subject to various graph theoretic constraints. In spectral graph theory, the maximization and minimization of the spectral radius of a given class of a graph and the problem of determining the extremal graphs find a particular interest among researchers. One such result due to Brualdi and Solheid in~\cite{Brualdi} is an important motivation for our study. In this article, the authors obtained the graphs that maximize the spectral radius amongst graphs with a fixed number of vertices in a given class of graphs. Later, several results in similar contexts has been published (for example see~\cite{Conde, Feng,Guo,Chu, Lui,Lu, Lou, Xiao, Xu}). For a few more interesting results on spectral graph theory reader also may refer~\cite{DS}.

The color-change rule on a graph $G$ is defined as follows. Suppose that $G$ is a graph with each vertex colored either white or blue.  If $u$ is a blue vertex in $G$ and exactly one neighbor $v$ of $u$ is white, then change the color of $v$ to blue, we say that $u$ forces $v$ and write $u \to v$. Given a graph $G$, a subset $S$ of vertices is called a zero forcing set for $G$ if it has the property that when initially the vertices in $S$ are colored blue and the remaining vertices are colored white, then applying the color-change rule all the vertices of $G$ become blue. The smallest size of a zero forcing set for $G$ is denoted by $Z(G)$ and is called the zero forcing number of $G$. A zero forcing set for $G$ of size $Z(G)$ is called a minimum zero forcing set of $G$.

The minimum rank of a graph $G$ is defined to be the smallest possible rank over all symmetric real matrices whose $ij-$th entry, for $i \ne j$ is nonzero whenever $i \sim j$ in $G$ and is zero otherwise. The minimum rank of $G$ is denoted by $\textup{mr}(G)$. The maximum nullity of $G$, denoted by $M(G)$ is defined as $M(G) = n - \textup{mr}(G)$. In~\cite{Bar}, it was shown that $M(G) \le Z(G)$ for any graph $G$. Zero forcing number of a graph has been extensively studied and has connections to inverse eigenvalue problem for graphs. Recently, the zero forcing number for graphs has been studied relating it with other graph parameters like connected domination number, degree sequence, path cover, perfect dominating set, chromatic number, to name a few. In a more recent paper~\cite{Zhang}, the authors find the extremal graph among the class of trees with zero forcing number at most $k$ that attains maximum spectral radius. This opens a new direction which connects spectral graph theory and zero forcing number of graphs. Motivated by the work done in~\cite{Zhang}, in this article we study the class of clique trees(which is a natural generalization of trees), where each block has size at least $3$ and a fixed zero forcing number to find the extremal graph where the maximum spectral radius is attained.

The article is organized as follows: In Section~\ref{sec:pre} we state some existing results about zero forcing number and prove some lemmas that will be essential for the proof of the main result of the article. Later, in Section~\ref{sec:main} we state and prove the main theorem of the article. Finally, we give an upper bound of the maximum spectral radius in terms of the zero forcing number.

\section{Preliminary Results}\label{sec:pre}

If a graph $G$ is a vertex-sum of two graphs $G_1$ and $G_2$ at $v$, then the following theorem gives us a relation between the minimum rank of $G$ and minimum rank of $G_1$ and $G_2$.

\begin{theorem}[Cut-vertex Reduction Theorem ~\cite{Bar0}]
	If $G = G_1 \oplus_v G_2$, then $\mr(G) = \min\{\mr(G_1)+\mr(G_2), \mr(G_1-v)+\mr(G_2-v)+2\}$.
\end{theorem}

From the cut-vertex reduction theorem, the following corollary follows if $G$ is a vertex-sum of two graphs $G_1$ and $G_2$ at $v$.
\begin{cor}\label{cor:v}~\cite{Huang}
	$M(G_1 \oplus_v G_2) = \max\{M(G_1)+M(G_2),M(G_1-v)+M(G_2-v)\} - 1$.
\end{cor}

It is well known that for any graph $G$, the maximum nullity of $G$ bounded above by the zero forcing number of $G$. The next theorem gives us a class of graphs where the equality is attained.

\begin{theorem}\label{thm:Z(G)=M(G)}~\cite{Huang}
	Let $G$ be a clique tree, then $Z(G) = M(G)$.
\end{theorem}

\begin{lem}\label{lem:Z(G)-v}~\cite{Huang}
	If $v$ is a vertex in a graph $G$, then $Z(G - v) -1 \le Z(G) \le Z(G-v)+1$.
\end{lem}

\begin{cor}\label{cor:Z(G)}
	Let $G$ be a clique tree such that each block is a clique of size at least three and $v$ is a cut vertex such that $K_m$ is a pendant block that contains $v$, then
	$$Z(G_1 \oplus_v K_m) = Z(G_1) + Z(K_m) -1,$$ where $G_1 = G-(K_m \setminus \{v\})$.
\end{cor}

\begin{proof}
	Using Corollary~\ref{cor:v} and Theorem~\ref{thm:Z(G)=M(G)} we have $$Z(G_1 \oplus_v K_m) =  \max\{Z(G_1)+Z(K_m),Z(G_1-v)+Z(K_m-v)\} - 1.$$ We know that $Z(K_m) = m-1$ and hence $Z(K_m-v) = m-2$. Note that if $Z(G_1-v) = Z(G_1) +1$, then we have $Z(G_1-v)+Z(K_m-v) = Z(G_1)+1+m-2 = Z(G_1)+m-1$. Thus it follows from Lemma~\ref{lem:Z(G)-v} that $Z(G_1-v)+Z(K_m-v) \le Z(G_1)+Z(K_m)$ and hence the result follows.
\end{proof}

The next theorem gives us a formula for the zero forcing number of $G$ in terms of the zero forcing number of its blocks and the proof follows from Corollary~\ref{cor:v}, but we provide the proof for completeness.

\begin{theorem}\label{thm:Z(G)}
	Let $G$ be a clique tree with $b$ blocks $B_1,B_2,\cdots,B_b$, where each block have size at least $3$ then the zero forcing number of $G$ is given by $$Z(G) = \sum_{i=1}^b Z(B_i) - \sum_{v \in V(G)} (bi_G(v) - 1)$$
\end{theorem}

\begin{proof}
	We prove the result by applying mathematical induction on the number of blocks \textit{i.e.} $b$. If $b =1$, then $G$ is a clique of size at least $3$ and all vertices have $bi(v) = 1$ and hence the result holds true. Next, we assume that the result is true for clique trees with $b-1$ blocks. Without loss of generality, we assume that $B_b$ is a pendent block and $G = H \oplus_w B_b$. Using Corollary~\ref{cor:Z(G)} we have $$Z(G) = Z(H) + Z(B_b) - 1.$$ Note that $H$ is a clique tree with $b-1$ blocks, where each block have size at least $3$. Thus using induction hypothesis we have that the result is true for $Z(H)$, \textit{i.e.} $$Z(H) = \sum_{i=1}^{b-1} Z(B_i) - \sum_{v \in V(H)} (bi_H(v) - 1).$$ It is easy to observe that 
	\begin{equation*}
		bi_H(v) = \left\{
		\begin{array}{ll}
			bi_G(v) & \text{for } v\ne w, \\
			bi_G(v) -1 & \text{for } v = w.\\
		\end{array}
		\right.
	\end{equation*} Thus, we have 
	\begin{align*}
		Z(G) &= Z(H) + Z(B_b) - 1\\
			 &= \sum_{i=1}^{b-1} Z(B_i) - \sum_{v \in V(H)} (bi_H(v) - 1) + Z(B_k) - 1\\
			 &= \sum_{i=1}^b Z(B_i) - \sum_{v \in V(G)} (bi_G(v) - 1).
	\end{align*} Hence the result follows for any clique tree where each block have size at least three by mathematical induction.
\end{proof}

\begin{rem}
	Let $G$ be a clique tree on $n$ vertices, zero forcing number $k$ and $b$ blocks, where each block has size at least $3$, then using Theorem~\ref{thm:Z(G)} we can conclude that $k = n - b$. Since each block has size at least $3$, we have $$\left \lfloor \frac{n}{2} \right \rfloor + 1 \le k \le n-1.$$ When $n$ is even, the clique tree which attains $k = \left \lfloor \frac{n}{2} \right \rfloor + 1$ has $\left \lfloor \frac{n}{2} \right \rfloor -2$ blocks of $K_3$ and one block of $K_4$ and when $n$ is odd, the clique tree has $\left \lfloor \frac{n}{2} \right \rfloor -1$ blocks of $K_3$.
\end{rem}

\begin{lem}\label{lem:Perron-vector}
	Let $A$ be the adjacency matrix of a clique tree $G$ with $(\rho(G),x)$ being the eigenpair corresponding to the Perron value. Let $m \ge 3$ and $K_m$ be a pendant block of $G$ with vertex set $\{v_1,\cdots,v_m\}$. If $v_1$ is the cut vertex, then $x_{v_1}> x_{v_i}$ for all $2 \le i \le m$ and $x_{v_i} = x_{v_j}$ for all $i,j \in \{2,3,\cdots, m\}$.
\end{lem}

\begin{proof}
Using the adjacency relation $Ax = \rho(G) x$, we have 
	\begin{equation}
		\rho(G) x_{v_i} = x_{v_1} + \left(\sum_{j=2}^{m} x_{v_j}\right) - x_{v_i}
	\end{equation}
for $2 \le i \le m$. Since the block $K_m$ is a pendant block, due to graph automorphism we can relabel the vertices $\{v_2,v_3,\cdots,v_m\}$ in any order keeping the graph structure invariant. Hence the entries corresponding to these vertices in a Perron vector are same, \textit{i.e.} $x_{v_i} = x_{v_j}$ for all $i,j \in \{2,3,\cdots,m\}$. If we assume $x_{v_i} = y$ for all $2 \le i \le m$, then from the adjacency relation we have
$$x_{v_1} = (\rho(G) -(m-2))y.$$ Since, $K_m$ is a subgraph of $G$ we have $\rho(G) > (m-1)$, where $\rho(K_m) = m-1$. Thus, we can conclude that $x_{v_1} > x_{v_i}$ for all $2 \le i \le m$.
\end{proof}

Let $G$ be a clique tree with one block of $K_m$, $m \ge 3$ and all the remaining blocks are $K_3$, which are attached to more than one vertex of $K_m$. Let $A$ be the adjacency matrix of $G$ and $(\rho(G),x)$ be an eigenpair corresponding to the Perron eigenvalue $\rho(G)$. Let the vertices of $K_m$ be labeled as $\{v_1,v_2,\cdots,v_m\}$ such that $x_{v_1} \ge x_{v_i}$ for all $ 1 \le i \le m$.

\begin{lem}\label{lem:K_n-K_3}
	Let $G$ be a clique tree defined as above. If $G^{*}$ is a graph obtained from $G$ by moving all block of $K_3's$ attached to different vertices of $K_m$ to a single vertex, namely $v_1$, then $\rho(G) < \rho(G^*)$. 
\end{lem}

\begin{proof}
	Let the vertices adjacent to $v_i$ for $i \ne 1$, which are not a part of $K_n$ be $\{u_1,u_2,\cdots,u_s\}$. Let $G'$ be the graph obtained from $G$ by performing the following graph transformations:
	\begin{itemize}
		\item[(a)] Delete the edges $v_i \sim u_j$ for all $ 1 \le j \le s$
		\item[(b)] Add the edges $v_1 \sim u_j$ for all $ 1 \le j \le s$.
	\end{itemize} Let $A'$ be the adjacency matrix of $G'$. Then, using the condition $x_{v_1} \ge x_{v_i}$, we have
	$$\frac{1}{2} x^T(A'-A)x = \sum_{j=1}^s x_{v_1}x_{u_j} -  \sum_{j=1}^s x_{v_i}x_{u_j} = \sum_{j=1}^s (x_{v_1}-x_{v_i})x_{u_j} \ge 0.$$
	Thus, we have $\rho(G) \le \rho(G')$. Applying the similar graph transformation repetitively we obtain the graph $G^*$ and using the same argument we have $\rho(G) \le \rho(G^*)$. Further, using Lemma~\ref{lem:Perron-vector}, a Perron vector of $G^*$ cannot be a Perron vector for $G$ and hence, we have $\rho(G) < \rho(G^*)$. 
\end{proof}

\section{Main Results}\label{sec:main}
Let $l,m \ge 4$ and $G$ be a clique tree that contains $H = K_l \oplus_v K_m$ as a subgraph. Let $A(G)$ be the adjacency matrix of $G$ and $(\rho(G),\textbf{x})$ be the eigen pair corresponding to the spectral radius $\rho(G)$. Let $p,q \in H \cap K_l \setminus \{v\}$ be two vertices such that $x_p,x_q$ is the least among all the vertices of $H \cap K_l \setminus \{v\}$. Without loss of generality we assume that $x_p \le x_q$. In this setup we have the following two lemmas.

\begin{lem}\label{lem:1}
	Let $r,s \in H \cap K_m \setminus \{v\}$  be two vertices such that $x_r,x_s \ge x_p,x_q$. If $G^*$ be the graph obtained from $G$ by the following operations:
	\begin{itemize}
		\item Delete the edges $u \sim p$ and $u \sim q$ for all $u \in K_l \setminus \{p,q,v\}$,
		\item Add the edges $u \sim w$ for all $u \in K_l \setminus \{p,q,v\}$ and $w \in K_m \setminus \{v\}$,
	\end{itemize} then $\rho(G^*) > \rho(G)$. Moreover $Z(G) = Z(G^*)$.
\end{lem}

\begin{proof}
	Let $A(G^*)$ be the adjacency matrix of $G^*$. Since $l,m \ge 4$, there exists at least one vertex in $K_l$ other than $p,q,v$ and the same holds true for $K_m$. Using the fact that $x_r,x_s \ge x_p,x_q$ we have
	\begin{align*}
		\frac{1}{2} \textbf{x}^T(A(G^*)-A(G))\textbf{x} &= - \sum_{u \in K_l \setminus \{p,q,v\}} x_u x_p - \sum_{u \in K_l \setminus \{p,q,v\}} x_u x_q + \sum_{u \in K_l \setminus \{p,q,v\}} \sum_{w \in K_m \setminus \{v\}} x_ux_w\\
		&> - \sum_{u \in K_l \setminus \{p,q,v\}} x_u x_p - \sum_{u \in K_l \setminus \{p,q,v\}} x_u x_q + \sum_{u \in K_l \setminus \{p,q,v\}} x_u x_r \\
		&~~~~+ \sum_{u \in K_l \setminus \{p,q,v\}} x_u x_s\\
		&= \sum_{u \in K_l \setminus \{p,q,v\}} x_u(x_r-x_p) + \sum_{u \in K_l \setminus \{p,q,v\}} x_u (x_s-x_q) \ge 0.
	\end{align*} Thus we have $\rho(G^*) > \rho(G)$. In the resulting graph $G^*$, the subgraph $H = K_l \oplus_v K_m$ has been replaced by $K_3 \oplus_v K_{l+m-3}$ and it using Corollary~\ref{cor:Z(G)} we have 
	\begin{align*}
		Z(K_l \oplus_v K_m) &= Z(K_l) + Z(K_m) -1 = l+m-3,\\
		Z(K_3 \oplus_v K_{l+m-3}) &= Z(K_3) + Z(K_{l+m-3}) = l+m-3.
	\end{align*} One can observe that $bi_G(u) = bi_{G^*}(u)$ for all vertices $u$. Thus using Theorem~\ref{thm:Z(G)} and the fact that $Z(K_l \oplus_v K_m) = Z(K_3 \oplus_v K_{l+m-3})$ we have $Z(G) = Z(G^*)$.
\end{proof}

\begin{lem}\label{lem:2}
	Let $r \in H \cap K_m \setminus \{v\}$  be a vertex such that $x_r \ge x_p,x_q$ and for all vertices $u \in  H \cap K_m \setminus \{v,r\}$ we have $x_u < x_q$. Let $s \in  H \cap K_m \setminus \{v,r\}$ be such that $x_s < x_q$. If $G^*$ be the graph obtained from $G$ by the following operations:
	\begin{itemize}
		\item Delete the edges $u \sim p$ for all $u \in K_l \setminus \{p,v\}$,
		\item Delete the edges $u \sim s$ for all $u \in K_m \setminus \{s,v\}$,
		\item Add the edge $p \sim s$,
		\item Add the edges $u \sim w$ for all $u \in K_l \setminus \{p,v\}$ and $w \in K_m \setminus \{s,v\}$,
	\end{itemize} then $\rho(G^*) > \rho(G)$. Moreover $Z(G) = Z(G^*)$.
\end{lem}

\begin{proof}
	Let $A(G^*)$ be the adjacency matrix of $G^*$. Since $l,m \ge 4$, there exists at least one vertex in $K_l$ and in $K_m$ other than $p,q,v$ and $r,s,v$, respectively. Using the fact that $x_r \ge x_p,x_q$ and $x_p > x_s$ for all $s \in  H \cap K_m \setminus \{v,r\}$ we have
	\begin{align*}
		\frac{1}{2} \textbf{x}^T(A(G^*)-A(G))\textbf{x} 
		&= - \sum_{u \in K_l \setminus \{p,v\}} x_u x_p - \sum_{u \in K_m \setminus \{s,v\}} x_u x_s + x_p x_s + \sum_{u \in K_l \setminus \{p,v\}} \sum_{w \in K_m \setminus \{s,v\}} x_u x_w\\
		&> - \sum_{u \in K_l \setminus \{p,v\}} x_u x_p - \sum_{u \in K_l \setminus \{s,v\}} x_u x_s + \sum_{u \in K_l \setminus \{p,v\}} x_u x_r + \sum_{u \in K_l \setminus \{s,v\}} x_u x_q\\
		&\qquad +x_px_s-x_qx_r\\
		&= \sum_{u \in K_l \setminus \{p,v\}} x_u(x_r-x_p) + \sum_{u \in K_l \setminus \{s,v\}} x_u (x_q-x_s) +x_px_s-x_qx_r\\
		&= \sum_{u \in K_l \setminus \{p,v\}} x_u(x_r-x_p) + \sum_{u \in K_l \setminus \{r,s,v\}} x_u (x_q-x_s) + x_r(x_q-x_s)\\
		&\qquad +x_px_s-x_qx_r\\
		&= \sum_{u \in K_l \setminus \{p,v\}} x_u(x_r-x_p) + \sum_{u \in K_l \setminus \{r,s,v\}} x_u (x_q-x_s) -(x_r-x_p)x_s\\
		&= \sum_{u \in K_l \setminus \{p,q,v\}} x_u(x_r-x_p) + \sum_{u \in K_l \setminus \{r,s,v\}} x_u (x_q-x_s) +(x_r-x_p)(x_q - x_s)\\
		&\ge 0.
	\end{align*}
	Thus, we have $\rho(G^*) > \rho(G)$. In the resulting graph $G^*$, the subgraph $H = K_l \oplus_v K_m$ has been replaced by $K_3 \oplus_v K_{l+m-3}$ and by similar argument as in Lemma~\ref{lem:1} we have $Z(G) = Z(G^*)$.
\end{proof}

Let $G(n,k)$ be the class of clique tree on $n$ vertices and zero forcing number $k$ with size of each block is at least $3$. Let $\mathcal{G}(n,k) \in G(n,k)$ be the graph with $n-k-1$ blocks of $K_3$ and a single block of $K_{2k-n+2}$ all attached to a single cut-vertex, \textit{i.e.} $$\mathcal{G}(n,k) = K_{2k-n+2} \circledcirc K_3^{(n-k-1)}.$$ For  reference one can see Figure~\ref{fig:extremal}. Note that, it is easy to observe that $Z(\mathcal{G}(n,k)) = k$.

\begin{figure}[ht]
	\centering
	\begin{tikzpicture}
		\draw[fill=none,thick](0,0) circle (2.0) node [black,yshift=0.0cm] {$K_{2k-n+2}$};
		\Vertex[x=0,y=-2,color=red,size=.2]{0}
		\Vertex[x=-1,y=-4,color=black,size=.2]{1}
		\Vertex[x=-2,y=-3,color=black,size=.2]{2}
		
		\Vertex[x=1,y=-4,color=black,size=.2]{3}
		\Vertex[x=2,y=-3,color=black,size=.2]{4}
		
		
		\Edge(0)(1)
		\Edge(0)(2)
		\Edge(0)(3)
		\Edge(0)(4)
		\Edge(2)(1)
		\Edge(3)(4)
		
		\node at (0,-5) {$n-k-1$};
		
		\draw [dots]  (-1,-4) -- (1,-4);
		
		\draw [decorate, 
		decoration = {brace, raise=5pt,	amplitude=10pt}] (2,-4) -- (-2,-4);
	\end{tikzpicture}
	\caption{$\mathcal{G}(n,k) = K_{2k-n+2} \circledcirc K_3^{(n-k-1)}$.} \label{fig:extremal}
\end{figure}
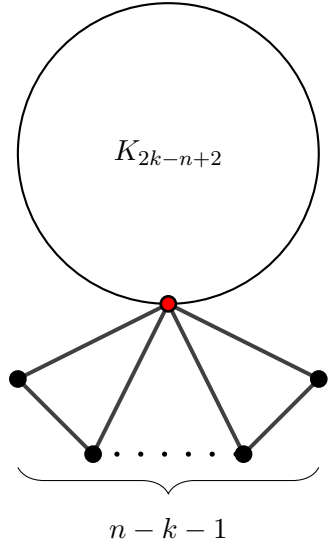

\begin{theorem}
	Let $G \in G(n,k)$ be a clique tree on $n$ vertices and zero forcing number $k$, then $$\rho(G) \le \rho(\mathcal{G}(n,k))$$ and the equality is attained if and only if $G \cong \mathcal{G}(n,k)$.
\end{theorem}

\begin{proof}
	Let $G \in G(n,k)$ be a clique tree on $n$ vertices and zero forcing number $k$ with maximal spectral radius. Then, using Lemmas~\ref{lem:1} and \ref{lem:2}, we can conclude that $G$ cannot have two blocks that have size at least $4$ that shares a common cut-vertex. Further, we claim that $G$ can have at most one cut-vertex. Suppose, for a contradiction, we assume that $G$ has more than one cut vertex. Then we have one of the following as a subgraph of $G$:
	\begin{itemize}
		\item[1.] A $K_3$ along with $K_m,K_l$ attached to two different vertices of $K_3$, where $l,m \ge 3$.
		\item[2.] A $K_m$ with $m \ge 4$ and $2$ $K_3's$ attached to two different vertices of $K_m$.
	\end{itemize}
	Let us consider the first case, where $u,v$ be two cut-vertices of $K_3$ with $K_m,K_l$  attached to $u,v$ respectively. Let $(\rho(G),\textbf{x})$ be the eigen pair corresponding to the spectral radius $\rho(G)$. Without loss of generality we can assume that $x_u \ge x_v$, then $G^*$ be the graph obtained by applying the graph transformation where we move the block $K_l$ from $v$ to $u$. Then we have $\rho(G^*) \ge \rho(G)$. If both $m$ and $n$ are $\ge 4$, then again using Lemmas~\ref{lem:1} and \ref{lem:2} we can further increase the spectral radius, thereby reaching a contradiction.
	
	Thus, the extremal graph in $G(n,k)$ can have at most one block of $K_l$, where $l \ge 4$ and the remaining blocks are $K_3$'s. Then using Lemma~\ref{lem:K_n-K_3} we can conclude that the extremal graph must have all the $K_3$'s  attached to a single vertex of the $K_l$. Hence we have $G \cong \mathcal{G}(n,k)$ and the result follows.
\end{proof}

The adjacency matrix of $\mathcal{G}(n,k)$ can be expressed in the following block form
\begin{equation}
	A(\mathcal{G}(n,k)) = \begin{bmatrix}
		(J-I)_{2k-n+1} & \mathds{1}_{2k-n+1} &\textbf{0}_{(2k-n+1)\times(2(n-k-1))}\\
		\mathds{1}_{2k-n+1}^T & 0 & \mathds{1}_{2(n-k-1)}^T\\
		\textbf{0}_{(2k-n+1)\times(2(n-k-1))}^T & \mathds{1}_{2(n-k-1)} & D
	\end{bmatrix},
\end{equation}
where $D$ is a block diagonal matrix of order $2(n-k-1)$ with $(J-I)_2$ as the $2 \times 2$ diagonal blocks. Observe that, each block of the adjacency matrix of $\mathcal{G}(n,k)$ has constant row sum, hence the equitable quotient matrix of $A(\mathcal{G}(n,k))$ is given by
\begin{equation*}
	Q(\mathcal{G}(n,k)) = \begin{bmatrix}
		2k-n & 1 & 0\\
		2k-n+1 & 0 & 2(n-k-1)\\
		0 & 1 & 1
	\end{bmatrix}.
\end{equation*}
The characteristics polynomial of $Q(\mathcal{G}(n,k))$ is given by
\begin{equation}\label{eqn:f(x)}
	f(x) = \det(xI -Q(\mathcal{G}(n,k))) = x^3 +(n-2k-1)x^2-(2n-2k-1)x+k-(n-k-1)(2n-2k+1).
\end{equation}
The vertex set of $\mathcal{G}(n,k)$ can be partitioned as follows
$$V(\mathcal{G}(n,k)) = \{v\} \cup \{v_1,v_2,\cdots, v_{2k-n+1}\} \cup \{u_1^{(1)},u_2^{(1)}\} \cup \cdots \cup \{u_1^{(n-k-1)},u_2^{(n-k-1)}\},$$ where $v$ is the central cut vertex, the vertex set of $K_{n-k+2}$ is $ \{v\} \cup \{v_1,v_2,\cdots, v_{2k-n+1}\}$ and $\{v\} \cup  \{u_1^{(i)},u_2^{(i)}\}$ is the vertex set of the $i^{th}$ $K_3$ for $1 \le i \le n-k-1$. Next, let us consider the following sets of vectors
$$E_1 = \{e(v_1) - e(v_j) : 2 \le j \le 2k-n+1\} \cup \{e(u_1^{(1)}) - e(u_2^{(1)})\} \cup \cdots \cup  \{e(u_1^{(n-k-1)}) - e(u_2^{(n-k-1)})\}$$ and
$$E_2 = \{e(u_1^{(1)}) + e(u_2^{(1)}) - e(u_1^{(i)}) - e(u_2^{(i)}) : 2 \le i \le n-k-1 \},$$ where $e(u)$ is the column vector with $1$ at the $u^{th}$ coordinate and $0$ at all other remaining coordinates. Now, it is easy to observe that $$A(\mathcal{G}(n,k)) x = -x \text{ for all } x \in E_1$$ and $$A(\mathcal{G}(n,k)) x = x \text{ for all } x \in E_2.$$ Thus, we can conclude that $-1$ is an eigenvalue of $A(\mathcal{G}(n,k))$ with multiplicity $(2k-n)+(n-k-1) = k-1$ and $1$ is an eigenvalue of $A(\mathcal{G}(n,k))$ with multiplicity $n-k-2$. One can observe that $1$ and $-1$ are not roots of the equation $f(x) = 0$, where $f(x)$ is defined in Eqn.~\eqref{eqn:f(x)}. Combining above observations leads us the following theorem.

\begin{theorem}
	The characteristics polynomial of the adjacency matrix of $\mathcal{G}(n,k)$ is given by
	$$g(x) = \det(xI-A(\mathcal{G}(n,k))) = (x-1)^{n-k-2} (x+1)^{k-1} f(x),$$ where $f(x)$ is the polynomial as defined in Eqn.~\eqref{eqn:f(x)}.
\end{theorem}

It is easy to see that $K_{2k-n+2}$ and $K_{1,n-1}$ are proper subgraphs of $\mathcal{G}(n,k))$ and hence we have $$\rho(\mathcal{G}(n,k))) > \max\{2k-n+1, \sqrt{n-1}\}.$$ Next, using Theorem~\ref{thm:equitable} we can conclude that the spectral radius of $\mathcal{G}(n,k))$ is the largest root of $f(x)$, as defined in Eqn.~\eqref{eqn:f(x)}. For simplicity of further calculations, let us assume $a = 2k-n+1$ and $b = \sqrt{n-1}$. Now, using $a$ and $b$, the function $f(x)$ in Eqn.~\eqref{eqn:f(x)} can be rewritten as
\begin{equation}\label{eqn:f(x)-new}
	f(x) = x^3 -ax^2-(n-a)x + k - (k-a)(n-a).
\end{equation}
Now let us consider the following two cases.\\

\noindent \underline{\textbf{Case 1:}} $\left \lfloor \frac{n}{2} \right \rfloor + 1 \le k \le \frac{1}{2} \left(n-1+\sqrt{n-1}\right).$\\

Let $\rho(\mathcal{G}(n,k))) = a + d$, where $d > 0$. Then, $d$ is the largest root of
\begin{equation}
	f(a+x) = x^3 + 2ax^2 + (a^2-(n-a))x + k - k(n-a).
\end{equation} 
Next, using the fact $d > 0$, we have $2ad^2 + (a^2-(n-a))d + k - k(n-a) < 0$ and hence
$$ d < \frac{n-a(a+1) + \sqrt{(a^2+a-n)^2 + 8ak(n-1-a)}}{4a}.$$
Thus, we can conclude that 
$$ \rho(\mathcal{G}(n,k))) < a + \frac{n-a(a+1) + \sqrt{(a^2+a-n)^2 + 8ak(n-1-a)}}{4a},$$ where $a = 2k-n+1$.\\

\noindent \underline{\textbf{Case 2:}} $\frac{1}{2} \left(n-1+\sqrt{n-1}\right) < k \le n-1$.\\

Let $\rho(\mathcal{G}(n,k))) = b + e$, where $e > 0$. Then, $e$ is the largest root of
\begin{equation}
	f(b+x) = x^3 + (3b-a)x^2 + (2n-2ab+a-3)x +(1-a)(a-b)+k-k(n-a).
\end{equation}

Let $\alpha = 2n-2ab+a-3$ and $\beta = (1-a)(a-b)+k-k(n-a)$. Next, using the fact $e > 0$, we have $(3b-a)e^2 + \alpha e + \beta < 0$ and hence
$$e < \frac{\sqrt{\alpha^2 - 4 \beta (3b-a)} -\alpha }{2(3b-a)}.$$
Thus, we can conclude that 
$$ \rho(\mathcal{G}(n,k))) < b + \frac{\sqrt{\alpha^2 - 4 \beta (3b-a)} -\alpha }{2(3b-a)},$$ where $a = 2k-n+1$ and $b = \sqrt{n-1}$.

Combining above two cases we give an upper bound for the spectral radius of $\mathcal{G}(n,k)$ which we present as a theorem below.

\begin{theorem}
	Let $\left \lfloor \frac{n}{2} \right \rfloor + 1 \le k \le n-1$ and $G \in G(n,k)$ be a clique tree. Then,
	\begin{align*}
		\rho(G) < \begin{cases}
			a + \dfrac{n-a(a+1) + \sqrt{(a^2+a-n)^2 + 8ak(n-1-a)}}{4a} & \textup{ if } \left \lfloor \frac{n}{2} \right \rfloor + 1 \le k \le \frac{1}{2} \left(n-1+\sqrt{n-1}\right)\\
			\\
			b + \dfrac{\sqrt{\alpha^2 - 4 \beta (3b-a)} -\alpha }{2(3b-a)} & \textup{ if }  \frac{1}{2} \left(n-1+\sqrt{n-1}\right) < k \le n-1
		\end{cases}
	\end{align*}
	where $a = 2k-n+1$, $b = \sqrt{n-1}$, $\alpha = 2n-2ab+a-3$ and $\beta = (1-a)(a-b)+k-k(n-a)$.\\
\end{theorem}

\noindent{\textbf{\Large Acknowledgements}}:The author is partially supported by the National Science and Technology Council in Taiwan (Grant ID: NSTC-111-2628-M-110-002). The author would also like to take the opportunity to thank Prof. Jephian Lin for carefully reading the article and giving valuable suggestions and comments.

\small{
}

\end{document}